\documentclass[11 pt]{article}
\usepackage{amsmath, amsthm, enumerate}
\usepackage{tikz}
\usepackage{hyperref}

\title{A Proof of a Conjecture of Ohba}
\author{Jonathan A. Noel\thanks{Mathematical Institute, University of Oxford, UK.} \and Bruce A. Reed\thanks{School of Computer Science, McGill University, Canada.} \and Hehui Wu\thanks{Department of Mathematics and Statistics, Simon Fraser University, Canada.}}

\newtheoremstyle{case}{}{}{\normalfont}{}{\itshape}{:}{ }{}

\newtheorem{thm}{Theorem}[section]
\newtheorem{lem}[thm]{Lemma}
\newtheorem{prop}[thm]{Proposition}
\newtheorem{conj}[thm]{Conjecture}
\newtheorem*{repeatConj}{Conjecture~\ref*{OhbaConj}}
\newtheorem{claim}[thm]{Claim}
\newtheorem{cor}[thm]{Corollary}

\theoremstyle{definition}
\newtheorem{defn}[thm]{Definition}
\newtheorem*{ack}{Acknowledgements}

\theoremstyle{remark}
\newtheorem{rem}[thm]{Remark}

\theoremstyle{case}
\newtheorem{case}{Case}

\numberwithin{equation}{section}

\begin{document}

\maketitle

\begin{abstract}
We prove a conjecture of Ohba which says that every graph $G$ on at most $2\chi(G)+1$ vertices satisfies $\chi_\ell(G)=\chi(G)$.
\end{abstract}


\section{Introduction}
List colouring is a variation on classical graph colouring. An instance of list colouring is obtained by assigning to each vertex $v$ of a graph $G$ a list $L(v)$ of \emph{available colours}. An \emph{acceptable colouring} for $L$ is a proper colouring $f$ of $G$ such that $f(v)\in L(v)$ for all $v\in V(G)$. When an acceptable colouring for $L$ exists, we say that $G$ is \emph{$L$-colourable}. The \emph{list chromatic number} $\chi_\ell$ is defined in analogy to the chromatic number:
\[\chi_\ell(G)=\min\{k: G\text{ is } L\text{-colourable whenever } |L(v)|\geq k \text{ for all } v\in V(G)\}.\]
List colouring was introduced independently by Vizing~\cite{Vizing} and Erd\H{o}s, Rubin and Taylor~\cite{ERT} and researchers have devoted a considerable amount of energy towards its study ever since (see e.g.~\cite{Alon,Tuza,Krat,Woodall}).

A graph $G$ has an ordinary $k$-colouring precisely if it has an acceptable colouring for $L$ where $L(v)=\{1,2,\dots,k\}$ for all $v\in V(G)$. Therefore, the following bound is immediate:
\[\chi\leq \chi_\ell.\]
At first glance, one might expect the reverse inequality to hold as well. It would seem that having smaller intersection between colour lists could only make it easier to map adjacent vertices to different colours. However, this intuition is misleading; in reality, a lack of shared colours can have the opposite effect. In fact, there are bipartite graphs with arbitrarily large list chromatic number, as was shown in the original paper of Erd\H{o}s et al.~\cite{ERT}. 

On the other hand, there are many special graph classes whose elements are conjectured to satisfy $\chi_\ell=\chi$; such graphs are said to be \emph{chromatic-choosable}~\cite{OhbaOrig}. Probably the most well known problem in this area is the List Colouring Conjecture, which claims that every line graph is chromatic-choosable. This was independently formulated by many different researchers, including Albertson and Collins, Gupta, and Vizing (see~\cite{HC}). Galvin~\cite{Galvin} proved that line graphs of bipartite graphs are chromatic-choosable, and Kahn~\cite{Kahn} proved that every line graph satisfies $\chi_\ell\leq (1+o(1))\chi$. Other classes of graphs which have been conjectured to satisfy $\chi=\chi_\ell$ include claw-free graphs~\cite{ClawFree}, total graphs~\cite{Total} and squares of graphs~\cite{Square}; the last of these conjectures was recently disproved by Kim and Park~\cite{KP}. In fact, it was shown independently in~\cite{KKP} and~\cite{Yeager} that there does not exist an integer $k$ such that $G^k$ is chromatic-choosable for all $G$, answering a question of Zhu~\cite{ZhuPower}. 

In~\cite{OhbaOrig}, Ohba proved that $\chi_\ell(G+K_n)=\chi(G+K_n)$ for any graph $G$ and sufficiently large $n$, where $G+H$ denotes the join of $G$ and $H$. In their original paper~\cite{ERT}, Erd\H{o}s et al. proved that the complete multipartite graph $K_{2,2,\dots,2}$ is chromatic-choosable and the same was proved for $K_{3,2,2,\dots,2}$ by Gravier and Maffray in~\cite{GravMaf}. This paper concerns a conjecture of Ohba~\cite{OhbaOrig}, which implies the last three results. 
\begin{conj}[Ohba~\cite{OhbaOrig}]
\label{OhbaConj}
If $|V(G)|\leq 2\chi(G)+1$, then $G$ is chromatic-choosable.
\end{conj}
Infinite families of graphs satisfying $|V(G)|=2\chi(G)+2$ and $\chi_\ell(G)>\chi(G)$ are exhibited in~\cite{Examples} and so Ohba's Conjecture is best possible with respect to the bound on $|V(G)|$. 

It is easy to see that the operation of adding an edge between vertices in different colour classes of a $\chi(G)$-colouring does not increase $\chi$ or decrease $\chi_\ell$. It follows that Ohba's Conjecture is true for all graphs if and only if it is true for complete multipartite graphs. Thus, we can restate Ohba's Conjecture as follows. 
\begin{repeatConj}[Ohba~\cite{OhbaOrig}]
\label{OhbaMulti}
If $G$ is a complete $k$-partite graph on at most $2k+1$ vertices, then $\chi_\ell(G)=k$. 
\end{repeatConj}

This conjecture has attracted a great deal of interest and substantial evidence has been amassed for it. This evidence mainly comes in two flavours: replacing $2k+1$ with a smaller function of $k$, or restricting to graphs whose stability number is bounded above by a fixed constant.

\begin{thm}
\label{known}
Let $G$ be a complete $k$-partite graph. If any of the following are true, then $G$ is chromatic-choosable.
\begin{enumerate}[(a)]
\item $|V(G)|\leq k+\sqrt{2k}$ \emph{(Ohba~\cite{OhbaOrig})};
\item $|V(G)|\leq \frac{5}{3}k -\frac{4}{3}$ \emph{(Reed and Sudakov~\cite{ReedSudakov})};
\item $|V(G)|\leq(2-\varepsilon)k-n_0(\varepsilon)$ for some function $n_0$ of $\varepsilon\in (0,1)$.  \emph{(Reed and Sudakov~\cite{ReedSudakov2})}.\footnote{Reed and Sudakov~\cite{ReedSudakov2} actually prove that there is a function $n_1(\varepsilon)$ such that if $n_1(\varepsilon)\leq |V(G)|\leq (2-\varepsilon)k$, then $G$ is chromatic-choosable. The original statement is equivalent to our formulation here.}
\end{enumerate}
\end{thm}

\begin{defn}
A maximal stable set of a complete multipartite graph is called a \emph{part}.
\end{defn}

\begin{thm}
\label{known2}
Let $G$ be a complete $k$-partite graph on at most $2k+1$ vertices and let $\alpha$ be the size of the largest part of $G$. If $\alpha\leq 5$, then $G$ is chromatic-choosable 
{\rm (Kostochka, Stiebitz and Woodall~\cite{Kostochka}; He et al.~\cite{Alpha=3} proved the result for $\alpha\leq 3$).}
\end{thm}

In this paper, we prove  Ohba's Conjecture. We divide the argument into three main parts. In Section~\ref{nasect}, we show how a special type of proper non-acceptable colouring of $G$ can be modified to yield an acceptable colouring for $L$. Then in Section~\ref{proc} we argue that, under certain conditions, it is possible to find a colouring of this type. Finally, in Section~\ref{add}, we complete the proof by showing that if Ohba's Conjecture is false, then there exists a counterexample which satisfies the conditions described in Section~\ref{proc}. 

For the proper non-acceptable colourings we consider in Section~\ref{nasect}, if $v$ is coloured with a colour $c$  not on $L(v)$, then we insist  that no other vertex is coloured with $c$, and that $c$ appears on the lists of many vertices.  This helps us to prove, in Section~\ref{nasect},  that we can  modify  such a colouring to obtain an acceptable colouring for $L$, as there  are at least $k$ colours which can be used on $v$ and many vertices on which $c$ can be used.

In Sections~\ref{proc} and~\ref{add}, we  combine  Hall's Theorem and various counting arguments to prove that such colourings exist for  a minimal counterexample to Ohba's Conjecture. In the rest of this section we provide some properties of such a minimal counterexample, which will help us do so. For one, we show that  for  any minimal counterexample  $G$, if  $G$ is not $L$-colourable,  then the total number of colours in the union of the lists of  $L$ is at most $2k$. This upper bound on the number of colours, foreshadowed in earlier work of Kierstead~\cite{Kierstead} and Reed and Sudakov~\cite{ReedSudakov2,ReedSudakov},  is crucial in that it  implies the existence of colours which appear in the lists of many vertices, which our approach requires.  

A variant of Ohba's Conjecture for on-line list colouring has been proposed~\cite{online, online2}. We remark that this problem remains open, since our approach, with its heavy reliance on Hall's Theorem, does not apply to the on-line variant.

\subsection{Properties of a Minimal Counterexample}
\label{minprop}

Throughout the rest of the paper we assume, to obtain a contradiction, that Ohba's Conjecture is false and we let $G$ be a minimal counterexample in the sense that $G$ is a complete $k$-partite graph on at most $2k+1$ vertices such that $\chi_\ell(G)>k$ and Ohba's Conjecture is true for all graphs on fewer than $|V(G)|$ vertices. Throughout the rest of the paper, $L$ will be a list assignment of $G$ such that $|L(v)|\geq k$ for all $v\in V(G)$ and $G$ is not $L$-colourable. Define $C:=\bigcup_{v\in V(G)}L(v)$ to be the set of all colours.

Let us illustrate some properties of a minimal counterexample. To begin, suppose that for a non-empty set $A\subseteq V(G)$ there is a proper colouring $g:A\to C$ such that \[g(v)\in L(v) \text{ for all }v\in A.\]
Such a mapping is called an \emph{acceptable partial colouring} for $L$. Define $G':=G-A$ and $L'(v):=L(v)-g(A)$ for each $v\in V(G')$. If for some $\ell\geq0$ the following inequalities hold, then we can obtain an acceptable colouring of $G'$ for $L'$ by minimality of $G$:
\begin{equation}\label{V} |V(G')|\leq 2(k-\ell)+1,\end{equation}
\begin{equation}\label{chi} \chi(G')\leq k-\ell,\text{ and}\end{equation}
\begin{equation}\label{L} |L'(v)|\geq k-\ell \text{ for all } v\in V(G'). \end{equation}
However, such a colouring would extend to an acceptable colouring for $L$ by colouring $A$ with $g$, contradicting the assumption that $G$ is not $L$-colourable. Thus, no such  $A$ and $g$ can exist. 

This argument can be applied to show that every part $P$ of $G$ containing at least two elements must have $\bigcap_{v\in P}L(v)=\emptyset$. Otherwise, if $c\in\bigcap_{v\in P}L(v)$, then the set $A:=P$ and function $g(v):=c$ for all $v\in A$ would satisfy (\ref{V}), (\ref{chi}) and (\ref{L}) for $\ell=1$, a contradiction. We have proven:

\begin{lem}
\label{cap}
If $P$ is a part of $G$ such that $|P|\geq2$, then $\bigcap_{v\in P}L(v)=\emptyset$.
\end{lem}

Many of our results are best understood by viewing an instance of list colouring in terms of a special bipartite graph. Let $B$ be the bipartite graph with bipartition $(V(G),C)$ where each $v\in V(G)$ is joined to the colours of $L(v)$. For $x\in V(G)\cup C$, we let $N_B(x)$ denote the neighbourhood of $x$ in $B$. Clearly a matching in $B$ corresponds to an acceptable partial colouring for $L$ where each colour is assigned to at most one vertex. Recall the classical theorem of Hall~\cite{Hall} which characterizes the sets in bipartite graphs which can be saturated by a matching.

\begin{thm}[Hall's Theorem~\cite{Hall}]
Let $B$ be a bipartite graph with bipartition $(X,Y)$ and let $S\subseteq X$. Then there is a matching $M$ in $B$ which saturates $S$ if and only if $\left|N_B(S')\right|\geq |S'|$ for every $S'\subseteq S$. 
\end{thm}

In the next proposition, we use the minimality of $G$ and Hall's Theorem to show that $B$ has a matching of size $|C|$. Slightly different forms of this result appear in the works of Kierstead~\cite{Kierstead} and Reed and Sudakov~\cite{ReedSudakov}.

\begin{prop}
\label{inj}
There is a matching in $B$ which saturates $C$.  
\end{prop}

\begin{proof}
Suppose to the contrary that no such matching exists. Then there is a set $T\subseteq C$ such that $\left|N_B(T)\right|<|T|$ by Hall's Theorem. Suppose further that $T$ is minimal with respect to this property. Now, for some $c\in T$ let us define $S:=T-c$ and $A:=N_B(S)$. By our choice of $T$, we observe that $\left|N_B(S')\right|\geq |S'|$ for every subset $S'$ of $S$. Thus, by Hall's Theorem there is a matching $M$ in $B$ which saturates $S$. Moreover, we have
\[|A|\geq|S|=|T|-1\geq\left|N_B(T)\right|\geq|A|.\]
This proves that $|A|=|S|$ and, since $N_B(T)$ is non-empty, it follows that $A$ is non-empty. In particular, $M$ must also saturate $A$. Let $g:A\to S$ be the bijection which maps each vertex of $A$ to the colour that it is matched to under $M$. Then clearly $g$ is an acceptable partial colouring for $L$. Since $A=N_B(S)$, every $v\in V(G)-A$ must have $L(v)\cap g(A)=L(v)\cap S=\emptyset$. Thus, $A$ and $g$ satisfy (\ref{V}), (\ref{chi}) and (\ref{L}) for $\ell=0$, a contradiction.
\end{proof}

Let us rephrase the above proposition into a form which we will apply in the rest of this section.

\begin{cor}
\label{injh}
There is an injective function $h:C\to V(G)$ such that $c\in L(h(c))$ for all $c\in C$.
\end{cor}

The following is a simple, yet useful, consequence of this result.

\begin{cor}
\label{<}
$|C|<|V(G)|\leq 2k+1$.
\end{cor}

\begin{proof} 
If $|V(G)|\leq |C|$, then the injective function $h:C\to V(G)$ from Corollary~\ref{injh} would be a bijection. The inverse of $h$ would be an acceptable colouring for $L$ since each colour $c\in C$ would appear on exactly one vertex of $G$ for which $c$ is available. This contradicts the assumption that $G$ is not $L$-colourable.  
\end{proof}

\begin{cor}
\label{disjoint}
If there are $u,v\in V(G)$ such that $L(u)\cap L(v)=\emptyset$, then $L(u)\cup L(v)=C$ and $|C|=2k$.
\end{cor}

\begin{proof}
Since the list of every vertex has size $k$, if two lists $L(u)$ and $L(v)$ are disjoint, then $|C|\geq |L(u)\cup L(v)|\geq2k$. However, by Corollary~\ref{<} we have $|C|<|V(G)|\leq 2k+1$ and so it must be the case that $L(u)\cup L(v)=C$ and $|C|=2k$.
\end{proof}

By Corollary~\ref{<}, the difference between $|V(G)|$ and $|C|$ is always positive. Throughout the rest of the paper, it will be useful for us to keep track of this quantity.

\begin{defn}
$\gamma:=|V(G)|-|C|$. 
\end{defn}

We conclude this section with two other useful consequences of Proposition~\ref{inj}.

\begin{cor}
\label{2k+1}
$|V(G)|=2k+1$.
\end{cor} 

\begin{proof}
If $G$ has a part of size 2, then, by  Lemma~\ref{cap}, the lists of the two vertices it contains are disjoint. Hence by Corollaries~\ref{<} and~\ref{disjoint}, $|C|=2k$ and $|V(G)|=2k+1$. 
 
Otherwise $G$ does not contain any part of size two and so $G$ must contain a singleton part, say $\{v\}$.\footnote{Otherwise, all parts of $G$ would contain at least $3$ elements and so $3k\leq |V(G)|\leq 2k+1$, which would imply $k\leq 1$. Ohba's Conjecture is trivially true in this case.} If $|V(G)|\leq 2k$, then we can obtain an acceptable colouring by using an arbitrary colour of $L(v)$ to colour $v$ and applying minimality of $G$. The result follows. 
\end{proof}

\begin{prop}
\label{surj}
If $f:V(G)\to C$ is a proper colouring, then there is a proper surjective colouring $g:V(G)\to C$ such that for every $v\in V(G)$, either
\begin{enumerate}[(a)]
\item  $g(v)\in L(v)$, or\label{inf}
\item  $g^{-1}(g(v))\subseteq f^{-1}(f(v))$. \label{outf}
\end{enumerate}
\end{prop}

\begin{proof}
Let $h:C\to V(G)$ be a function as in Corollary~\ref{injh}. Given a proper colouring $g:V(G)\to C$ and a colour $c\in C$, we say that $g$ \emph{agrees} with $h$ at $c$ if $g(h(c))=c$. 

Now, let $g:V(G)\to C$ be a proper colouring in which every vertex $v\in V(G)$ satisfies either (\ref{inf}) or (\ref{outf}) and, subject to this, the number of colours $c\in C$ at which $g$ agrees with $h$ is maximized. We show that $g$ is surjective. Otherwise, let $c'\in C-g(V(G))$ be arbitrary and define a colouring $g':V(G)\to C$ as follows:
\[g'(v)=\left\{\begin{array}{ll}	c' 		& \text{if }v=h(c'),\\
																	g(v)	&	\text{otherwise}.\end{array}\right.\]
																	
Clearly $g'$ is proper since $g$ does not map any vertex to $c'$. Moreover, $g'$ agrees with $h$ at $c'$ and at every colour at which $g$ agrees with $h$. Let us show that every vertex $v$ of $G$ satisfies either (\ref{inf}) or (\ref{outf}) for $g'$, which will contradict our choice of $g$ and complete the proof.

In the case that $v=h(c')$, then we have $g'(v)=c'\in L(v)$ and so (\ref{inf}) is satisfied for $v$. Now, suppose that $v\neq h(c')$ and $g'(v)\notin L(v)$. Since every vertex $w\neq h(c')$ satisfies $g'(w)=g(w)\neq c'$, we see that 
\[g'^{-1}(g'(v))=g^{-1}(g(v))-h(c')\subseteq f^{-1}(f(v))\]
and so (\ref{outf}) is satisfied for $v$. The result follows. 
\end{proof}

\section{Bad Colourings With Good Properties}
\label{nasect}

In this section, we show that certain types of non-acceptable colourings can be modified to produce acceptable colourings. The following definitions describe the types of colourings that we are interested in. Say that a vertex $v\in V(G)$ is a \emph{singleton} if $\{v\}$ is a part of $G$. Recall that $\gamma = |V(G)|-|C|>0$.

\begin{defn}
A colour $c\in C$ is said to be 
\begin{itemize}
\item\emph{globally frequent} if it appears in the lists of at least $k+1$ vertices of $G$. 
\item\emph{frequent among singletons} if it appears in the lists of at least $\gamma$ singletons of $G$. 
\end{itemize}
If $c$ is either globally frequent or frequent among singletons, then we say that $c$ is \emph{frequent}.
\end{defn}

\begin{defn}
A proper colouring $f:V(G)\to C$ is said to be \emph{near-acceptable} for $L$ if for every vertex $v\in V(G)$ either
\begin{itemize}
\item $f(v)\in L(v)$, or
\item $f(v)$ is frequent and $f^{-1}(f(v))=\{v\}$.
\end{itemize}
\end{defn}

Suppose that $f$ is a proper colouring of $G$ and let $V_f:=\left\{f^{-1}(c): c\in C\right\}$ be the set of colour classes under $f$. Generalizing the construction of $B$ in Section~\ref{minprop}, we define $B_f$ to be a bipartite graph with bipartition $\left(V_f,C\right)$ where each colour class $f^{-1}(c)\in V_f$ is joined to the colours of $\bigcap_{v\in f^{-1}(c)}L(v)$.\footnote{One can think of this construction as taking the graph $G$ and collapsing each colour class of $f$ into a single vertex. Each collapsed vertex is then assigned a list which is the intersection of the lists of all vertices in its corresponding colour class.} A matching in $B_f$ corresponds to a partial acceptable colouring for $L$ whose colour classes are contained in $V_f$. We use this observation to prove the following.

\begin{lem}
\label{na}
If there is a near-acceptable colouring for $L$, then there is an acceptable colouring for $L$.
\end{lem}

\begin{proof}
Suppose that there is a near-acceptable colouring $f$ for $L$. A matching in $B_f$ which saturates $V_f$ would indicate an acceptable colouring for $L$ with the same colour classes as $f$. Therefore, we assume that no such matching exists. By Hall's Theorem, there is a set $S\subseteq V_f$ such that $\left|N_{B_f}(S)\right|<|S|$. 

Since $\left|N_{B_f}(S)\right|<|S|$ there must be a colour class $f^{-1}(c^*)\in S$ such that $c^*\notin N_{B_f}(S)$. In particular, we have $c^*\notin N_{B_f}(f^{-1}(c^*))$. It follows that there is a vertex $v$ such that $f(v)=c^*$ and $c^*\notin L(v)$. Since $f$ is near-acceptable for $L$, we must have that $c^*$ is frequent and $f^{-1}(c^*)=\{v\}$.

\begin{case}
 $c^*$ is globally frequent.
\end{case}

Since $f^{-1}(c^*)=\{v\}\in S$, we have that $N_{B_f}(S)\supseteq N_{B_f}(f^{-1}(c^*))= L(v)$ and so $\left|N_{B_f}(S)\right|\geq k$. This implies that
\begin{equation}\label{k+1}|S|>\left|N_{B_f}(S)\right|\geq k.\end{equation}

However, since $c^*\notin N_{B_f}(S)$, every colour class of $S$ must contain a vertex whose list does not contain $c^*$. Since $c^*$ is frequent, there are at most
\[|V(G)|-(k+1)\leq k\]
such vertices. Thus, $|S|\leq k$, contradicting (\ref{k+1}) and so the result holds in this case.

\begin{case}
\label{freqsing}
 $c^*$ is frequent among singletons. 
\end{case}

In order to complete the proof in this case, we impose some additional conditions on $f$ and $S$. 
\begin{itemize}
\item By Proposition~\ref{surj}, we may assume that $f$ maps surjectively to $C$.
\item We can assume that $S$ is chosen to be a set which maximizes $|S| - \left|N_{B_f}(S)\right|$ over all subsets of $V_f$.
\end{itemize}
By our choice of $S$, for any $T \subseteq V_f-S$ we must have $\left|N_{B_f}(T)-N_{B_f}(S)\right| \ge |T|$. It follows, by Hall's Theorem, that
\begin{equation}\label{match} \text{there is a matching $M$ in $B_f -N_{B_f}(S)$ saturating $V_f-S$.}\end{equation}
If $x$ is a singleton such that $c^*\in L(x)$, then $\{x\}$ is a colour class of $f$ which cannot be in $S$ since $f$ is proper and $c^* \not \in N_{B_f}(S)$. Since $c^*$ is frequent among singletons, we get that 
\begin{equation}\label{singgg}
\text{$V_f-S$ contains at least $\gamma$ singleton parts of $G$.}\end{equation}
The proof follows immediately from the following claim, which will be used again later in the paper.

\begin{claim}
\label{useagain}
Suppose that $f:V(G)\to C$ is surjective. If (\ref{match}) and (\ref{singgg}) hold, then there is an acceptable colouring for $L$. 
\end{claim}

\begin{proof}
We let $\ell$ be the number of  colour classes under $f$ with more than one element. Define $A$ to be the union of  the colour classes of $f$ which  either are  not in $S$ or contain more than one element and define $G':=G-A$.  We will find a partial acceptable colouring $g$ of $A$ satisfying  (\ref{V}), (\ref{chi}) and (\ref{L}) for  this definition of $\ell$. 

Since $A$ contains all colour classes of $f$ with more than one element, we have that $A$ contains at least $2\ell$ vertices and so (\ref{V}) holds for $\ell$. By (\ref{singgg}), we get that $\chi(G-A)\leq k-\gamma$. The fact that $f$ is surjective implies $\ell\leq \gamma$ and so (\ref{chi})  holds for $\ell$.  In order to show that (\ref{L}) holds, we will insist that our partial acceptable colouring satisfies the following:
\begin{equation*}\label{star}\tag{$*$}g(A)\text{ contains at most }\ell\text{ colours of }N_{B_f}(S).\end{equation*} 
For every vertex $w\in V(G)-A$, we see that $\{w\}$ is a colour class of $f$ contained in $S$; hence $L(w) \subseteq N_B(S)$. Thus, if (\ref{star}) holds, then (\ref{L}) holds for $\ell$. 

Therefore, to obtain a contradiction, we need only show that there is a partial acceptable colouring of $A$ satisfying (\ref{star}). To begin we note that the matching $M$ in (\ref{match}) defines a partial acceptable colouring $h$ for the set $A_1$ of  vertices  in the colour classes of $V_f-S$ using only colours in $C-N_{B_f}(S)$.  We also note that $f$ is a partial acceptable colouring when restricted to the vertices in the colour classes in $S$ with more than one element, since if $f(v) \not\in L(v)$ then $f^{-1}(f(v))=\{v\}$.  We define $g$ so that it agrees with $h$ on $A_1$ and agrees with $f$ on the rest of $A$. It is a partial acceptable colouring because, by definition,  
$g(A_1)$ is disjoint from $N_{B_f}(S)$ while  $g(A-A_1)$ is  contained in $N_{B_f}(S)$.  Furthermore, since $A-A_1$ consists of the union of at most $\ell$ colour classes of $f$, we see that (\ref{star}) holds for $g$ and we are done. 
\end{proof}
\end{proof}

\section{Constructing Near-Acceptable Colourings}
\label{proc}

In the previous section, we saw that finding an acceptable colouring for $L$ is equivalent to finding a near-acceptable colouring for $L$. However, in practice it can be much easier to construct a near-acceptable colouring than an acceptable colouring. In constructing a near-acceptable colouring, we need not worry about whether a frequent colour $c$ is available for a vertex $v$, provided that $v$ is the only vertex to be coloured with $c$. In this section, we exploit this flexibility to prove the following result.

\begin{lem}
\label{kfreqenough}
If $C$ contains at least $k$ frequent colours, then there is a near-acceptable colouring for $L$. 
\end{lem}

\begin{proof}
Let $F$ be a set of $k$ frequent colours and assume, to the contrary, that there does not exist a near-acceptable colouring for $L$. Our goal is to construct a near-acceptable colouring for $L$ by applying a three phase greedy procedure. In the first phase, choose a subset $V_1\subseteq V(G)$ and an acceptable partial colouring $f_1: V_1\to C-F$ such that $V_1$ contains as many vertices as possible and, subject to this, $V_1$ contains vertices from as many parts as possible. Before moving on, we prove the following claim. 

\begin{claim}
\label{2hit}
Every part of size two contains a vertex of $V_1$. 
\end{claim}

\begin{proof}
Suppose that $P=\{u,v\}$ is a part of size $2$ such that $P\cap V_1=\emptyset$. Then $L(u)\cap L(v)=\emptyset$ by Lemma~\ref{cap}, and so by Corollary~\ref{disjoint} we see that $L(u)\cup L(v)=C$ and $|C|=2k$. In particular, the image of $f_1$ must contain every colour $c\in C-F$ for, if not, we could use $c$ to colour one of $u$ or $v$, increasing the size of $V_1$. Since $|C|=2k$, we have $|C-F|=k$ which implies that $|V_1|\geq k$. 

If $|V_1|\geq k+1$, then $|V(G)-V_1|\leq |V(G)|-(k+1) =k=|F|$, and we can obtain a near-acceptable colouring for $L$ by mapping the vertices of $V(G)-V_1$ injectively to $F$. So, we must have $|V_1|=k$ and that every colour of $C-F$ is used by $f_1$ on a unique vertex of $V_1$. Since neither vertex of $P$ is in $V_1$ and $G$ has precisely $k$ parts, there must be a part $Q\neq P$ containing at least two vertices of $V_1$, say $x$ and $y$. However, since $L(u)\cup L(v)=C$, we can uncolour $x$ and use its colour to colour one of $u$ or $v$, which maintains the number of coloured vertices and increases the number of parts with a coloured vertex. This contradicts our choice of $V_1$ and completes the proof of the claim. 
\end{proof}

For each part $P$ let $R_P:=P-V_1$, the set of vertices which are not coloured by $f_1$. Label the parts of $G$ as $P_1,\dots,P_k$ so that $\left|R_{P_1}\right|\geq \dots\geq\left|R_{P_k}\right|$. The second phase of our colouring procedure is described as follows. For each part $P_i$, in turn, we try to colour $R_{P_i}$ with a frequent colour which has not yet been used and is available for every vertex of $R_{P_i}$. We terminate this phase when we reach an index $i$ for which this is not possible. Define
\[V_2:=\bigcup_{j=1}^{i}R_{P_j},\text{ and}\]
\[V_3:=\bigcup_{j=i+1}^{k}R_{P_j}\]
That is, $V_2$ is the set of vertices coloured by the second phase and $V_3$ is the set of vertices which must be coloured in the third. If $i=k$, then we have obtained an acceptable colouring for $L$ after the first two phases and we are done. So, we assume that $i<k$ and that there is no frequent colour which has not yet been used and is available for every vertex of $R_{P_{i+1}}$.

Let $U$ denote the set of colours of $F$ which have not been used in the second phase. We observe that $|U|=k-i$. If $|V_3|\leq k-i$, then in the third phase we simply map $V_3$ injectively into $U$, thereby obtaining a near-acceptable colouring for $L$. So, we can assume that $|V_3|\geq k-i+1$. In particular, this implies $\left|R_{P_{i+1}}\right|\geq2$ by our choice of ordering. Thus, again by our choice of ordering, we get that $|V_2|\geq \left|R_{P_{i+1}}\right|i \geq 2i$. Since $|V(G)|= 2k+1$ and $|V_3|\geq k-i+1$, we get
\[|V_1\cup V_2| = |V(G)|-|V_3| \leq (2k+1) - (k-i+1) = k+i.\] 
It follows that $|V_1|\leq k+i-|V_2|\leq k+i-2i=k-i$. 

Let us show that $|V_1|$ is exactly $k-i$. To do so, we use the fact that every colour in $U$ is absent from $L(v)$ for at least one $v\in R_{P_{i+1}}$. Since there are $k$ colours available for each vertex of $R_{P_{i+1}}$ and exactly $k$ colours in $F$, these absences imply that the colours of $C-F$ must appear at least $|U|=k-i$ times in the lists of vertices of $R_{P_{i+1}}$. Now if a colour $c\in C-F$ is available for $j>0$  vertices of $R_{P_{i+1}}$, then:
\begin{enumerate}[(i)]
\item\label{noti+1} $c$ was not used to colour any vertex of $P_{i+1}$. Otherwise, we could use $c$ to colour those $j$ vertices of $R_{P_{i+1}}$ for which it is available. This would contradict our choice of $V_1$. 
\item \label{>=j}at least $j$ vertices are coloured with $c$ in the first phase. Otherwise, we could uncolour the vertices that were coloured with $c$ and use $c$ to colour $j$ vertices of $R_{P_{i+1}}$ instead, again contradicting our choice of $V_1$. 
\end{enumerate}

Thus, since the colours of $C-F$ appear at least $k-i$ times in the lists of vertices in $R_{P_{i+1}}$, by (\ref{noti+1}) and (\ref{>=j}) we have that at least $k-i$ vertices of $V(G)-P_{i+1}$ were coloured in the first phase; that is, $|V_1|\geq k-i$. Since we have already proven that $|V_1|\leq k-i$, this implies  that
\begin{equation}\label{k-i}|V_1|=k-i, \text{ and}\end{equation}
\begin{equation}\label{no i+1}V_1\cap P_{i+1}=\emptyset.\end{equation}

Recall that, by Claim~\ref{2hit}, every part of size two intersects $V_1$. So, by (\ref{no i+1}), we must have $\left|R_{P_{i+1}}\right|=|P_{i+1}|\geq3$. By our choice of ordering, this implies that $|V_2|\geq \left|R_{P_{i+1}}\right|i\geq3i$, and so $|V_1|\leq k+i-3i=k-2i$. Thus, by (\ref{k-i}), we must have $k-2i\geq k-i$ which implies $i=0$. 

Therefore, we have $|V_1|=k$ and $|V_2|=0$, which implies that $U=F$, $|V_3|=k+1$, and $R_{P_1}=P_1$ by (\ref{no i+1}). At this point, our goal is to show that there exists a colour $c\in F$ which is available for two vertices $u,v\in P_1$. If such a colour exists, then, in the third phase of our procedure, we simply colour $u$ and $v$ with $c$ and map the vertices of $V_3-\{u,v\}$ to $F-c$ bijectively to obtain a near-acceptable colouring for $L$.

Recall that the number of times the colours of $C-F$ appear in lists of vertices in $P_1$ is at most the cardinality of $V_1$, which is exactly $k$. Since $|P_1|\geq3$ and each list has size at least $k$, the colours of $F$ must appear at least $2k$ times in the lists of vertices in $P_1$. Since $|F|=k$, there is a colour in $F$ (in fact, many) which is available for at least two vertices in $P_1$. This completes the proof. 
\end{proof}

\section{Adding Colours to the Lists}
\label{add}

From now on, we impose an additional  requirement on our list assignment $L$ for which there is no acceptable colouring. We insist that it is  maximal in the sense that increasing the size of any list makes it possible to find an acceptable colouring. That is, for any $v\in V(G)$ and $c\in C-L(v)$, if we define $L^*(v)=L(v)\cup\{c\}$ and $L^*(u)=L(u)$ for all $u\neq v$, then there is an acceptable colouring for $L^*$. Given this property, it is straightforward to prove that every frequent colour is available for every singleton. 

\begin{lem}
\label{common}
If $c\in C$ is frequent, then $c\in L(v)$ for every singleton $v$ of $G$. 
\end{lem}

\begin{proof}
Otherwise, add $c$ to the list of $v$. Since $L$ is maximal, there is an acceptable colouring for this modified list assignment. Since $G$ is not $L$-colourable, this colouring must use $c$ to colour $v$ and, since $v$ is a singleton, $v$ is the only vertex coloured with $c$. Therefore, this colouring is a near-acceptable colouring for $L$ and so by Lemma~\ref{na} it follows that $G$ is $L$-colourable, a contradiction. 
\end{proof}

Recall by Lemmas~\ref{na} and~\ref{kfreqenough} that there are fewer than $k$ frequent colours. We show now that this implies that there are at least $\gamma$ singletons. 

\begin{defn}
Let $b$ denote the number of non-singleton parts of $G$.
\end{defn}

\begin{prop}
\label{gammasing}
$G$ contains at least $\gamma$ singletons.  
\end{prop}

\begin{proof}
Suppose to the contrary that the number of singletons, namely $k-b$, is less than $\gamma$. Let $F'$ denote the set of all globally frequent colours. Then each colour of $C-F'$ is available for at most $k$ vertices, and each colour of $F'$ is available for at most $|V(G)|-b\leq 2k+1-b$ vertices by Lemma~\ref{cap}. Thus,
\[k|V(G)|\leq\sum_{v\in V(G)}|L(v)|=\sum_{c\in C}\left|N_B(c)\right|\]
\[\leq k|C-F'|+(2k+1-b)|F'|=k|C|+(k+1-b)|F'|.\]
In other words,
\begin{equation}\label{glob}|F'|\geq \frac{k(|V(G)|-|C|)}{k+1-b}=\frac{k\gamma}{k+1-b}.\end{equation}
Now, since we are assuming that $k+1-b\leq \gamma$, we obtain $|F'|\geq k$ by the above inequality. However, this contradicts the fact that there are fewer than $k$ frequent colours. Thus, $G$ must contain at least $\gamma$ singletons.
\end{proof}

\begin{cor}
\label{gamma+b}
$\gamma+b\leq k$.
\end{cor}

\begin{proof}
By Proposition~\ref{gammasing}, we have $k-b\geq \gamma$, which implies the result. 
\end{proof}

\begin{cor}
\label{freqequiv}
A colour $c\in C$ is frequent if and only if it is available for every singleton.
\end{cor}

\begin{proof}
By Proposition~\ref{gammasing}, any colour which is available for every singleton is automatically frequent among singletons, and therefore frequent. The reverse implication follows from Lemma~\ref{common}.
\end{proof}

By Lemma~\ref{kfreqenough} and Corollary~\ref{freqequiv}, to obtain the desired contradiction, we need only show that there are $k$ colours which are available for every singleton. In fact, as the next result shows, it is enough to prove that there are at least $b$ such colours. 

\begin{lem}
\label{btok}
There are fewer than $b$ frequent colours.
\end{lem}

\begin{proof}
Suppose that there are at least $b$ frequent colours and let $A_b=\{c_1,c_2,\dots,c_b\}$ be a set of $b$ such colours. By Lemma~\ref{common}, all such colours are available for every singleton. Label the singletons of $G$ as $v_{b+1},v_{b+2},\dots,v_{k}$. For each $i\in\{b+1,b+2,\dots,k\}$, in turn, choose a colour $c_i\in L(v_i)-A_{i-1}$ greedily and define $A_i:=A_{i-1}\cup\{c_i\}$. Let $L'$ be a list assignment of $G$ defined by
\[L'(v):=\left\{\begin{array}{ll}	A_{k} 		& \text{if } v \text{ is a singleton}, \\
																	L(v) 	& \text{otherwise}.\end{array}\right.\]
Clearly $L'$ assigns each singleton the same list of $k$ colours. Hence, there are at least $k$ frequent colours under $L'$ and so by Lemmas~\ref{na} and~\ref{kfreqenough} there is an acceptable colouring $f'$ for $L'$. We use this to construct an acceptable colouring $f$ for $L$, contradicting the fact that $G$ is not $L$-colourable. 

We let  $\mathcal{S}$ denote the set of all singletons and for each $v\in V(G)-\mathcal{S}$ we set $f(v)=f'(v)$. We note that $f'(\mathcal{S})$ is a set of exactly $k-b$ colours of $A_k$ disjoint from $f'(V(G)-\mathcal{S})=f(V(G)-\mathcal{S})$. We let \[\mathcal{S}':=\{v_i\in \mathcal{S}:c_i\in f'(\mathcal{S})\}.\]
We note that $|\mathcal{S}'| = k-b-|f'(\mathcal{S})\cap A_b|$ and hence $|\mathcal{S}-\mathcal{S'}| = |f'(\mathcal{S})\cap A_b|$. For each $v_i\in \mathcal{S}'$, we set $f(v_i)=c_i$. We arbitrarily choose a bijection $\pi:\mathcal{S}-\mathcal{S}'\to f'(S) \cap A_b$ and set $f(v)=\pi(v)$ for every $v\in \mathcal{S}-\mathcal{S}'$. Since each colour in $A_b$  is available for every singleton, we are done.
\end{proof}

Before moving on to the final counting arguments, we establish a few simple consequences of Proposition~\ref{gammasing}. 

\begin{cor}
\label{size2}
$G$ contains no part of size 2. 
\end{cor}

\begin{proof}
If $G$ contains a part $P=\{u,v\}$, then $L(u)\cap L(v)=\emptyset$ by Lemma~\ref{cap}. By Corollary~\ref{disjoint}, this implies that $|C|=2k$ and so $\gamma=1$. Thus, by Proposition~\ref{gammasing}, $G$ contains a singleton $v$ and every colour of $L(v)$ is frequent among singletons since $\gamma=1$. This contradicts the fact that there are fewer than $k$ frequent colours, and proves that $G$ contains no part of size two.
\end{proof}

\begin{cor}
\label{k+1/2}
$b\leq \frac{k+1}{2}$.
\end{cor}

\begin{proof}
Since $G$ contains no parts of size 2, we see that $G$ consists of $k-b$ singletons and $b$ parts of size at least three. Therefore
\[(k-b)+3b\leq |V(G)|\leq 2k+1\]
which implies that $b\leq\frac{k+1}{2}$.
\end{proof}

Recall that the set $F'$ of globally frequent colours satisfies
\[\frac{k\gamma}{k+1-b}\leq |F'| \leq b-1 < \frac{k}{2}\]
by (\ref{glob}), Lemma~\ref{btok} and Corollary~\ref{k+1/2}. This implies that
\begin{equation}\label{new1}2\gamma < k+1-b.\end{equation}
We will apply this inequality later on in the proof.

\subsection{A Bit of Counting} 

The final step is to apply a counting argument to show that there are at least $b$ colours that are frequent among singletons, which would contradict Lemma~\ref{btok} and complete the proof of Ohba's Conjecture.  In order to do so, we will find a fairly large set $X$ of singletons such that $N_B(X)$ is fairly small. This implies that the average number of singletons for which a colour in $N_B(X)$ is available is large. Using this, we will show that there are at least $b$ colours in $N_B(X)$ which are available for $\gamma$ singletons, which gives us the desired contradiction. We begin with the following proposition.

\begin{prop}
\label{strong}
Let $c^*$ be a colour which is not available for every singleton. Then there is a set $X$ (depending on $c^*$) of singletons such that
\begin{enumerate}[(a)]
\item \label{Xgamma}$|X|\geq k+1-b-\gamma$, and
\item\label{NX} $\left|\bigcup_{v\in X}L(v)\right|\leq2k-\left|N_B(c^*)\right|$.
\end{enumerate}
\end{prop}

\begin{proof}
To prove this result, we modify a colouring which is almost acceptable as in the proof of Lemma~\ref{na}. To begin, we let  $x$ be a singleton such that $c^*\notin L(x)$ and define a list assignment $L^*$ by
\[L^*(v)=\left\{\begin{array}{ll} 	L(x)\cup\{c^*\} 	& \text{if } v=x,\\
																	L(v)						& \text{otherwise}.\end{array}\right.\]

Then, since $L$ is maximal there is an acceptable colouring $f$ for $L^*$. Clearly $f^{-1}(c^*)=\{x\}$ and for every $v\in V(G)-x$ we have $f(v)\in L(v)$. By Proposition~\ref{surj}, we can also assume that $f$ is surjective. If there is a matching in $B_f$ which saturates $V_f$, then $G$ is $L$-colourable. Thus, we obtain a set $S\subseteq V_f$ such that $\left|N_{B_f}(S)\right|<|S|$. We can choose $S$ to maximize $|S|-\left|N_{B_f}(S)\right|$. By this choice, we have that (\ref{match}) is satisfied.

Since $\left|N_{B_f}(S)\right|<|S|$ and $c$ is adjacent to $f^{-1}(c)$ in $B_f$ for every $c\neq c^*$ it must be the case that $f^{-1}(c^*)\in S$ and $c^*\notin N_{B_f}(S)$. Since $c^*\notin N_{B_f}(S)$, every colour class of $S$ must contain a vertex whose list does not contain $c^*$. Thus, we have
\[\left|N_{B_f}(S)\right|<|S|\leq |V(G)|-\left|N_B(c^*)\right|\leq2k+1-\left|N_B(c^*)\right|.\]
Now, define $X$ to be the set of all singletons of $G$ whose colour classes under $f$ belong to $S$. Then $\bigcup_{v\in X}L(v)\subseteq N_{B_f}(S)$, and so (\ref{NX}) holds. 

Finally, if $X$ contains fewer than $k+1-b-\gamma$ singletons, then there are at least $\gamma$ singletons $y$ such that $\{y\}\in V_f-S$. That is, (\ref{singgg}) is satisfied. However, if this were the case, then we would obtain an acceptable colouring for $L$ by Claim~\ref{useagain}. Thus, (\ref{Xgamma}) must hold. 
\end{proof}

We let $c^*$ be a colour which is not frequent (equivalently, not available for every singleton), and subject to this maximizes $\left|N_B(c^*)\right|$. Let $X$ be a set of singletons as in Proposition~\ref{strong}.

Let $Z$ be a set of $b-1$ colours that appear most often in the lists of vertices of $X$ and define $Y:=N_B(X)-Z$. That is, $|Z|=b-1$ and if $c_1\in Z$ and $c_2\in Y$, then $\left|N_B(c_1)\cap X\right|\geq \left|N_B(c_2)\cap X\right|$. We can further assume that every frequent colour appears in $Z$, since by Lemma~\ref{common} these colours appear in the list of every vertex of $X$ and by Lemma~\ref{btok} there are at most $b-1$ of them. Finally, let $c'\in Y$ such that $\left|N_B(c')\cap X\right|$ is maximized. Our goal is to prove that $\left|N_B(c')\cap X\right|\geq\gamma$, contradicting the fact that $Z$ contains every frequent colour. The next proposition provides one way of doing this. 

\begin{defn}
$\beta:= k- \left|N_B(c^*)\right|$.
\end{defn}

\begin{rem}
\label{beta>0}
Since $c^*$ is not frequent, we have $\beta\geq0$. 
\end{rem}

\begin{prop}
\label{c'c*}
If $\beta\leq 2(k+1-b-2\gamma)$,  then $\left|N_B(c')\cap X\right|\geq \gamma$. 
\end{prop}

\begin{proof}
Since $|Z|=b-1$, every vertex $v\in X$ must satisfy $|L(v)\cap Y|\geq k-(b-1) = k+1-b$. Thus, by our choice of $c'$,
\[|Y|\left|N_B(c')\cap X\right|\geq\sum_{c\in Y}\left|N_B(c)\cap X\right|=\sum_{v\in X}|L(v)\cap Y|\geq |X|(k+1-b).\]
Combining this inequality with the two bounds of Proposition~\ref{strong} gives
\[\left|N_B(c')\cap X\right|\geq \frac{|X|(k+1-b)}{|Y|}= \frac{|X|(k+1-b)}{\left|N_B(X)\right|+1-b}\]
\begin{equation}\label{hmm}\geq\frac{(k+1-b-\gamma)(k+1-b)}{2k-\left|N_B(c^*)\right|+1-b}=\frac{(k+1-b-\gamma)(k+1-b)}{\beta + k +1 -b}\end{equation}
However, if $\beta\leq 2(k+1-b-2\gamma)$ then, by (\ref{new1}) and Remark~\ref{beta>0}, we have $0\leq\beta\gamma < (k+1-b-2\gamma)(k+1-b)$. So, multiplying top and bottom by $\gamma>0$, the right hand side of (\ref{hmm}) is at least
\[\frac{\gamma(k+1-b-\gamma)(k+1-b)}{(k+1-b-2\gamma)(k+1-b)+ \gamma(k +1 -b)}=\gamma.\]
The result follows. 
\end{proof}

Together with Proposition~\ref{c'c*}, the following proposition completes the proof of Ohba's Conjecture (with a factor of four to spare). 

\begin{prop}
$\beta<\frac{1}{2}(k+1-b-2\gamma)$.
\end{prop}

\begin{proof}
Let $F$ denote the set of all frequent colours. Recall that $c^*$ was chosen to maximize $\left|N_B(c^*)\right|$ over all colours which are not frequent. Thus, each colour $c\notin F$ must have $\left|N_B(c)\right|\leq \left|N_B(c^*)\right|$. Moreover, by Lemma~\ref{btok} there are at most $b-1$ frequent colours and by Lemma~\ref{cap} every colour $c\in C$ satisfies $\left|N_B(c)\right|\leq |V(G)|-b$. We have $|V(G)|=2k+1$ by Corollary~\ref{2k+1} and so,
\[(2k+1)k\leq \sum_{v\in V(G)}|L(v)|=\sum_{c\in C}\left|N_B(c)\right|\leq |C-F|\left|N_B(c^*)\right| + |F|(2k+1-b)\]
\[\leq (2k+1-\gamma-b+1)\left|N_B(c^*)\right| + (b-1)(2k+1-b)\]
since $|C|=|V(G)|-\gamma$ and $|F|\leq b-1$ by Lemma~\ref{btok}. Substituting $\left|N_B(c^*)\right|=k-\beta$ and rearranging, we obtain
\[(2k+2-\gamma-b)\beta\leq(-\gamma - b+1)k + (b-1)(2k+1-b)\]
\[=(b-1)(k+1-b)-k\gamma\]
By Corollary~\ref{gamma+b} we have $\gamma+b<2k+2$. Thus, we can divide both sides of the above inequality by $(2k+2-\gamma-b)$ to obtain
\[\beta\leq \frac{(b-1)(k+1-b) -k\gamma}{2k+2-\gamma-b}<\frac{\frac{1}{2}k(k+1-b-2\gamma)}{2k+2-\gamma-b}\]
since $b-1<\frac{k}{2}$ by Corollary~\ref{k+1/2}. Now, by Corollary~\ref{gamma+b}, we obtain
\[\frac{\frac{1}{2}k(k+1-b-2\gamma)}{2k+2-\gamma-b} \leq \frac{\frac{1}{2}k(k+1-b-2\gamma)}{k+2}< \frac{1}{2}(k+1-b-2\gamma)\]
which implies the result.
\end{proof}

This completes the proof of Ohba's Conjecture.

\section{Conclusion}

We conclude the paper by mentioning some subsequent work and open problems for future study. In general, one may ask the following: \emph{given a function $f(k)>2k+1$, what is the best bound on $\chi_\ell(G)$ for $k$-chromatic graphs on at most $f(k)$ vertices?} By applying the result of the current paper, Noel et al.~\cite{NWWZ} have solved this problem for every function $f$ such that $f(k)\leq 3k$ and $f(k)-k$ is even for all $k$. Their main result is the following strengthening of Ohba's Conjecture, which holds for all graphs $G$:
\begin{equation}\label{NWWZ}\chi_\ell(G)\leq\max\left\{\chi(G),\left\lceil\frac{|V(G)|+\chi(G)-1}{3}\right\rceil\right\}.\end{equation}
Ohba~\cite{3Ohba} proved that equality holds in (\ref{NWWZ}) whenever $G$ is a complete multipartite graph in which every part has size $1$ or $3$, generalizing a result of Kierstead~\cite{Kierstead} for $K_{3,3,\dots,3}$. Putting this together, we see that the list chromatic number of a graph on at most $3\chi(G)$ vertices is bounded above by the list chromatic number of the complete $\chi(G)$-partite graph in which every part has size $3$. Noel~\cite{Noel} has conjectured that a similar result holds for graphs on at most $m\chi(G)$ vertices for every fixed integer $m$. 

\begin{conj}[Noel~\cite{Noel}]
\label{m*k}
For $m,k\geq 2$, let $G$ be a $k$-chromatic graph on at most $mk$ vertices. Then $\chi_\ell(G)$ is bounded above by the list chromatic number of the complete $k$-partite graph in which every part has size $m$. 
\end{conj}

Noel~\cite{Noel} also proposed the following, more general, conjecture. 

\begin{conj}[Noel~\cite{Noel}]
\label{n/k}
For $n\geq k\geq 2$, there exists a graph $G_{n,k}$ such that
\begin{itemize}
\item $G_{n,k}$ is a complete $k$-partite graph on $n$ vertices,
\item $\alpha\left(G_{n,k}\right)=\left\lceil\frac{n}{k}\right\rceil$, and
\item if $G$ is a $k$-chromatic graph on $n$ vertices, then $\chi_\ell(G)\leq \chi_\ell\left(G_{n,k}\right)$. 
\end{itemize}
\end{conj}

Another, rather ambitious, problem could be to characterize all complete $k$-partite graphs with $\chi_\ell(G)=k$. Short of this, it would be interesting to characterize all such graphs on at most $f(k)$ vertices for some function $f$ which is larger than $2k+1$. In~\cite{Noel}, it is conjectured that if $G$ is a complete $k$-partite graph on $2k+2$ vertices, then $\chi_\ell(G)>k$ if and only if $k$ is even and either every part of $G$ has size $1$ or $3$, or every part of $G$ has size $2$ or $4$. The fact that such graphs are not chromatic-choosable was proved in~\cite{Examples}.

\begin{ack}
The authors would like to thank an anonymous referee for comments which helped to improve the exposition of this paper. 
\end{ack}

\bibliography{OhbaPaper}
  \bibliographystyle{amsplain}
\end{document}